\documentclass[11pt]{amsart}

\usepackage{ amssymb }
\usepackage{amsfonts}
\usepackage{tikz}

\usepackage{enumerate}
\usepackage{amsmath,amscd}
\usepackage{amsmath,amsthm,amssymb,anysize}
\usepackage[numbers,sort&compress]{natbib}
\usepackage{color}
\usepackage{graphicx}
\usepackage[all,cmtip]{xy}
\usepackage[colorlinks,citecolor=blue]{hyperref}
\usepackage{tikz-cd} 

\usepackage[normalem]{ulem}

\theoremstyle{definition}
\newtheorem{thm}{Theorem}
\newtheorem{lem}[thm]{Lemma}

\newtheorem{prop}[thm]{Proposition}
\newtheorem{cor}[thm]{Corollary}
\newtheorem{conj}[thm]{Conjecture}

\newtheorem*{defn}{Definition}
\newtheorem*{remark}{Remark}
\newcommand{\bea}{\begin{eqnarray}}
\newcommand{\eea}{\end{eqnarray}}
\newcommand{\nn}{\nonumber}

\marginsize{3cm}{3cm}{2.5cm}{2.2cm}%左右上下
\numberwithin{equation}{section}
\numberwithin{thm}{section}
\def\Z{\mathbb{Z}}

\def\C{\mathbb{C}}

\def\F{\mathbb{F}}
\def\H{\mathrm{H}}

\def\B{\mathrm{B}}

\def\PU{\mathrm{PU}_3}
\def\PGL{\mathrm{PGL}_3(\C)}
\def\ASL{{\mathrm{ASL}_2(\F_3)}}
\def\SL{\mathrm{SL}}
\def\GL{\mathrm{GL}}
\def\U{\mathrm{U}}

\def\CP{\mathbb{CP}}

\def\A{\mathbb{A}}
\def\gm{\Gamma}
\def\go{\Gamma_{0}}
\def\pt{\widetilde{P}}
\def\He{\mathcal{H}}
\def\pe{\text{FindFlex}(\epsilon)}

\newcommand{\x}{\mathcal{X}} 
\newcommand{\xf}{\widetilde{\mathcal{X}}_\text{flex} }

\title{Topological complexity of finding flex points on cubic plane curves}

\author{Weiyan Chen}
\address{Weiyan Chen, Yau Mathematical Sciences Center, Ning Zhai, Tsinghua University, Haidian District, Beijing 100084, China.}
\email{chwy@tsinghua.edu.cn}

\author{Zheyan Wan}
\address{Zheyan Wan, Yanqi Lake Beijing Institute of Mathematical Sciences and Applications, Hefangkou Village, Huaibei Town, Huairou District, Beijing 101408, China.}
\email{wanzheyan@bimsa.cn}


\address{}
\email{}

\date{}

\begin{document}

%\date{}
\maketitle

\vskip 10pt

\begin{abstract}
We prove a lower bound for the topological complexity, in the sense of Smale, of the problem of finding a flex point on a cubic plane curve. The key is to bound the Schwarz genus of a cover associated to this problem. We also show that our lower bound for the complexity is close to be the best possible. 
\end{abstract}

\section{Introduction}

% This paper is a first step towards a broader program of understanding the topological complexity of enumerative problems. We focus on the problem of finding a flex point on a cubic plane curve. 

% Smale \cite{Smale} introduced the notion of topological complexity and studied the topological complexity of the problem of finding roots of a polynomial. In this paper, we study the topological complexity of another enumerative problem, namely, the problem of finding flex points on a cubic plane curve.

Smale \cite{Smale} introduced the notion of topological complexity and proved a lower bound for the topological complexity of the problem of finding roots of a polynomial. Smale's lower bounds were later improved by many authors, such as Vassiliev \cite{Vassiliev},  De Concini-Procesi-Salvetti \cite{DPS}, and Arone \cite{Arone}. In this paper, we study the topological complexity of another enumerative problem: the problem of finding flex points on a cubic plane curve. 

% Galois observed that it is impossible to solve a polynomial equation of degree 5 or higher by radicals because its Galois group is not solvable. Jordan \cite{Jordan} noticed that many enumerative problems in algebraic geometry also have Galois groups which control solvability in radicals. Harris \cite{Harris} revisited Jordan's works and computed the Galois groups of many classical enumerative problems. 

% In a different direction, Smale \cite{Smale} was interested in the solvability of polynomial equations by algorithms. He introduced the concept of topological complexity and proved a lower bound for the topological complexity of any algorithm that finds roots of a polynomial. Smale's lower bound was later improved by many authors, such as Vassiliev \cite{Vassiliev},  De Concini-Procesi-Salvetti \cite{DPS}, and Arone \cite{Arone}. Notice that the concept of topological complexity also applies to other enumerative problems, just as Galois group does. Inspired by the works of Jordan \cite{Jordan} and Harris \cite{Harris} in extending Galois theory to enumerative problems, we are interested in the topological complexity, à la Smale, of enumerative problems. In this paper, we focus on the problem of finding flex points on a cubic plane curve.

A cubic plane curve in $\CP^2$ is the  locus of a homogeneous polynomial $F(x,y,z)$ of degree 3. A flex point on a plane curve is where the curve intersects its tangent line with multiplicity $\ge3$. Every smooth cubic plane curve has 9 distinct flex points. For $\epsilon>0$, we consider the following problem
% \begin{equation}
%     \label{problem}
%     \text{$\pe$: given a smooth cubic plane curve, find a point in $\CP^2$ that is within $\epsilon$ distance from a flex point of the given curve. }
% \end{equation}
\begin{quote}
$\pe$: given a smooth cubic plane curve, find a point in $\CP^2$ that is within $\epsilon$ distance from a flex point of the given curve. 
\end{quote}
%More precisely, given a smooth cubic plane curve defined by $F$, we want to find a point $p\in \CP^2$ that is within $\epsilon$ distance from a flex point of $F$. 
Here distance is by the Fubini-Study metric on $\CP^2\cong S^5/S^1$ induced by the round metric on $S^5$. Following Smale \cite{Smale}, we consider an \emph{algorithm} to be a finite rooted tree, consisting of a root for the input, leaves for the output, and internal nodes of the following two types:
    \begin{quote}
    \begin{center}
    \emph{computation nodes}\ \ \begin{tikzpicture}[scale=0.3]
    \draw (0,0) circle(.2);
    \draw[->] (0,1)-- (0,0.1);
    \draw[->] (0,-0.1)-- (0,-1);
    \end{tikzpicture}
    \ \ \ and\ \ \ \ \  
    \emph{branching nodes}\ \  \begin{tikzpicture}[scale=0.3]
    \draw (0,0) circle(.2);
    \draw[->] (0,1)-- (0,0.1);
    \draw[->] (-0.08,-0.08)-- (-0.8,-0.8);
    \draw[->] (0.08,-0.08)-- (0.8,-0.8);
    \end{tikzpicture}\ \ .
        \end{center}
\end{quote}
% \begin{itemize}
%     \item \emph{Computation nodes}, \begin{tikzpicture}[scale=0.3]
   
%     \draw (0,0) circle(.1);
%     \draw[->] (0,1)-- (0,0.1);
%     \draw[->] (0,-0.1)-- (0,-1);
   
%     \end{tikzpicture}
    
%     \item \emph{Branching nodes}, \begin{tikzpicture}[scale=0.3]
   
%     \draw (0,0) circle(.1);
%     \draw[->] (0,1)-- (0,0.1);
%     \draw[->] (-0.08,-0.08)-- (-0.8,-0.8);
%     \draw[->] (0.08,-0.08)-- (0.8,-0.8);
%     \end{tikzpicture}
% \end{itemize}
An algorithm is said to \emph{solve} the problem $\pe$ if it inputs a cubic plane curve and outputs a point satisfying the requirement.  The \emph{topological complexity} of an algorithm is the number of branching nodes in the tree. The {topological complexity of the problem $\pe$} is the minimum of the topological complexity of all algorithms solving the problem. We defer more details to Section \ref{tree}. 

% We say an algorithm \emph{solves} the problem $\pe$ if it inputs the coefficients of $F(x,y,z)$ defining a cubic plane curve and outputs coordinates of a point satisfying the requirement. The inputs and outputs are in terms of real and imaginary parts. The \emph{topological complexity} of an algorithm is the number of branching nodes in the tree. The {topological complexity of the problem $\pe$} is the minimum of the topological complexity of all algorithms solving the problem.

\begin{thm}
    \label{top complexity}
    For any $\epsilon<\pi/3$, the topological complexity of $\pe$ is at least 7.
\end{thm}

% What do we mean by ``complexity"? We follow S. Smale \cite{Smale} who considered a similar problem of finding all roots of a complex polynomial $f(z)$ of a fixed degree. Smale defined the \emph{topological complexity} of a problem to be the minimum complexity of any algorithm that solves the problem. For Smale, an algorithm is a rooted tree; its complexity is the number of branching nodes in the tree. The topological complexity of such a problem is bounded from below by Schwarz genus, defined by A. Schwarz, of certain covering map associated to the problem. 

% We are interested in measuring how complex it is to solve enumerative problems. One way to measure such complexity, first introduced by Smale, is via the topological complexity of algorithms. The lower bound for topological complexity is given by the Schwarz genus, defined by Albert Schwarz.

A key tool for bounding the topological complexity is the Schwarz genus of covers, introduced by A. Schwarz \cite{Schwarz}. 

\begin{defn}
The \emph{Schwarz genus} of a cover $E\xrightarrow[]{\xi} B$ is the minimal $k$ such that $B$ can be covered by $k$ many open subsets on which the projection map $\xi$ has a continuous section. Notation in this paper: $g({\xi})$ or $g(E/B).$ 
\end{defn}

% A cubic plane curve in $\CP^2$ is the zero locus of a homogeneous  polynomial $F(x,y,z)$ of degree 3. Every smooth cubic plane curve has 9 distinct flex points. Let $\x$ denote the parameter space of all smooth cubic plane curves:
% In this paper, we are interested in bounding the Schwarz genus of covers given by flex points on smooth cubic plane curves. 

We will focus on the following covering map. Define
\begin{align*}
    &\x := \big\{ F(x,y,z)\ |\ F \textrm{ is a nonsingular homogeneous polynomial of degree 3}  \big\} / \C^\times\\
    &\xf := \big\{ (F,p)\in \x\times\CP^2\ | \ p  \textrm{ is a flex point on the curve defined by } F \big\}.
\end{align*}
In particular $\x=\CP^9-\{\text{singular curves}\}$. The map $\xf\to \x$ given by $(F,p)\mapsto F$ is a 9-fold covering. The main topological theorem of the paper is the following:

\begin{thm}
\label{main}
    $8\le g(\xf/\x)\le 9.$
\end{thm}
We will prove the lower and the upper bounds in the theorem  using different methods. Proving the lower bound is the most difficult part of the proof where we will use Leary's computation of the cohomology of certain $p$-groups \cite{Leary}. By a standard argument, the Schwarz genus minus one gives a lower bound for the topological complexity. Hence, Theorem \ref{main} implies Theorem \ref{top complexity} above.
%The upper bound in Theorem \ref{main} tells us that the lower bound we obtained above is close to be the best possible.
% because any flex-finding algorithm with $k$ many output leaves, or equivalently, $k-1$ many branching nodes, will give 

% \begin{thm}
% \label{upper bound}
%     $g(\xf/\x)\le 9$
% \end{thm}

Whether $g(\xf/\x)$ is equal to 8 or 9 remains open. In the last section of the paper, we show that $g(\xf/\x)=8$ if a particular cohomology class $\mathfrak{o}_8$ vanishes. Moreover, we prove that $3\mathfrak{o}_8=0$ in Proposition \ref{3-torsion}. This leads us to make the following conjecture.
\begin{conj}
    $g(\xf/\x)= 8.$
\end{conj}

\begin{remark}[\textbf{Broader context}]
Our work is part of a broader program: to understand how complex it is to solve enumerative problems in algebraic geometry. Here ``solve" can have various meanings, according to which ``solvability" will be determined by different invariants.

For example, Galois discovered that a polynomial equation of degree 5 or higher cannot be solved by radicals because its Galois group is not solvable. In his 1870 treatise on Galois theory \cite{Jordan}, Jordan  noticed that many enumerative problems in algebraic geometry also have Galois groups which control solvability in radicals. However, the study of Galois group in enumerative geometry remained dormant until Harris' 1979 paper ``Galois groups of enumerative problems" \cite{Harris}. 

% Harris \cite{Harris} revisited Jordan's treatment and computed the Galois groups of many classical enumerative problems. 

Hilbert asked the the following question known as the Hilbert's thirteenth problem:  How far can the solution to  a degree 7 polynomial equation be simplified? See \cite{FW} for a modern formulation of the problem. In 1975, Brauer \cite{Brauer} introduced the notion of resolvent degree, which determines how far the solution to a polynomial equation can be simplified. Brauer's work was mostly forgotten (never cited) until 2018, when Farb-Wolfson \cite{FW} revisited Brauer's treatment. More importantly, Farb-Wolfson proposed the research program to study the resolvent degree of enumerative problems in general. 

Our work is inspired by the aforementioned works of Harris and Farb-Wolfson. We aim to do what they did but for topological complexity: use this invariant which was originally defined for polynomial equations to study other enumerative problems. These three invariants - Galois groups, resolvent degree, and topological complexity - provide different ways to measure how complex it is to solve enumerative problems, whether the solutions are by radicals, by algebraic functions in fewer variables, or by algorithms. 
\end{remark}

\subsection*{Acknowledgement}
We would like to thank Jintai Ding, Benson Farb, Wouter van Limbeek, Craig Westerland, and Jesse Wolfson for helpful conversations. WC first heard about the problem of bounding topological complexity of enumerative problems from Benson Farb and Jesse Wolfson in a class which they co-taught at the University of Chicago in 2016. WC particularly thanks Wouter van Limbeek for his suggestion which eventually points to a proof of Proposition \ref{K upper bound} below. WC is partially supported by the Young Scientists Fund of the National Natural Science Foundation of China (Grant No. 12101349).

\section{Previous results}
In this section, we recalled several standard results that we will use later in our proofs. None of the results in this section is due to us.

\subsection{General results about Schwarz genus}
% We first recall some general results 
All results mentioned in this subsection were proven by Schwarz in \cite{Schwarz}. 

% It is straightforward from definition that Schwarz genus is nonincreasing under pullbacks.
\begin{prop}[\cite{Schwarz}, page 71]\label{pullback SG}
Let $i^*\xi$ denote the pullback of a cover $\xi:E\to B$ along a continuous map $i: B'\to B$. Then $g(i^*\xi)\le g(\xi)$.
% $\xymatrix{A\ar[r]\ar[d]_h&X\ar[d]^f\\
% B\ar[r]^i&Y}$
% is a pullback of $f$ along $i$. Then $g(h)\le g(f)$.
\end{prop}
% This is because any section of $\xi$ gives a section of its pullback $i^*\xi$.

\begin{thm}[\cite{Schwarz}, page 71]\label{fiberwise join}
Consider a covering 
$F\to E\xrightarrow{\xi} B$. Let $\xi^{*k}$ denote the fiberwise join bundle $F^{*k}\to E_k\xrightarrow{\xi^{*k}}B$ where $F^{*k}$ is the topological join of $F$ with itself $k$ times.
Then $g(\xi)\le k$ if and only if $\xi^{*k}$ has a continuous section.
% $F^{*k}=\{\sum_{i=1}^k t_ip_i\mid p_i\in F, t_i\in[0,1], \sum_{i=1}^k t_i=1\}$.
\end{thm}

\begin{cor}[\cite{Schwarz}, page 76]
\label{coho dim}
If $\xi:E\to B$ is a covering where $B$ is a CW complex of dimension $d$, then $g(\xi)\le d+1.$
\end{cor}
\begin{proof}[Proof Theorem \ref{fiberwise join} $\Rightarrow$ Corollary \ref{coho dim}]
By obstruction theory, the obstruction to extending a section of $\xi^{*(d+1)}$ to its $i$-skeleton is a cohomology class in $\H^{i+1}(B,\pi_i(F^{*(d+1)}))$. We know the cohomology group vanishes for $i\ge d = \dim B$. If $i\le d-1$, then $\pi_i(F^{*(d+1)})=0$ since $F^{*(d+1)}$ is always $(d-1)$-connected. There is no obstruction to a section of $\xi^{*(d+1)}$.
\end{proof}

% Any Galois $\Gamma$-cover $\xi:Y\to X=Y/\Gamma$ is isomorphic to the pullback of the universal normal $\Gamma$-cover $E\Gamma\to B \Gamma$ via the classifying map $cl:X\to B\Gamma$.

% % Then 
% % $\xymatrix{X\ar[r]\ar[d]_f&EG\ar[d]\\
% % Y\ar[r]^{cl}&BG}$,

% \begin{defn}
% For any $\Gamma$-module $A$, the \emph{homological $A$-genus} of a Galois $\gm$-cover $Y\xrightarrow[]{\xi} X$ is 
% $$h_A(\xi):=\min\{k:\H^i(B\Gamma;A)\xrightarrow{cl^*}\H^i(X;A)\text{ is zero for any }i\ge k\}.$$
% \end{defn}

\begin{thm}[\cite{Schwarz}, page 98]\label{homo genus}
Suppose that $\xi:Y\to X$ is a Galois $\Gamma$-cover with a classifying map $cl:X\to B\Gamma$ where $B\gm$ is the classifying space of the discrete group $\gm$. %such that $\xi$ is isomorphic to the pullback of the universal Galois $\Gamma$-cover $E\Gamma\to B \Gamma$ via the classifying map $cl:X\to B\Gamma$. 
If 
$$\H^i(B\Gamma;A)\xrightarrow{cl^*}\H^i(X;A) \text{ is nonzero for some $i$ and some $\gm$-module $A$,}$$
then $g(\xi)\ge i+1$.
%We have
% $$g(\xi)\ge \min\{k:\H^i(B\Gamma;A)\xrightarrow{cl^*}\H^i(X;A)\text{ is zero for any }i\ge k\}.$$
% The right hand side is called the \emph{homological $A$-genus} of $\xi$, denoted by $h_A(\xi)$.
\end{thm}

\subsection{The Hesse pencil and the Hessian group}\label{section Hesse}
Now we recall several classical facts about cubic plane curves. We refer to Section 2 and Section 4 of a paper by Artebani and Dolgachev \cite{Artebani} for proofs and further discussions on these results. 

Every smooth cubic plane curve is projectively equivalent to a curve of the form:
\begin{equation}
    \label{pencil}
    F_\lambda : x^3+y^3+z^3-3\lambda xyz=0
\end{equation}
for some $\lambda\in \C$ such that $\lambda^3\ne 1$. This 1-parameter family of cubic plane curves $F_\lambda$ is called the \emph{Hesse pencil}. %Those 9 points form the famous \emph{Hesse configuration} of 12 lines in $\CP^2$. 
% See Figure \ref{Hesse config} below for the incidence relation among those points and lines. % insert a picture
% \begin{figure}[h]
%     \centering
%     \includegraphics[scale=.2]{"Hesse-configuration".png}
%      \caption{The Hesse configuration of 9 points and 12 lines (Need to redraw!!!)}
%      \label{Hesse config}
% \end{figure}
The \emph{Hessian group} is the subgroup $\gm$ of $\mathrm{Aut}(\CP^2)=\PGL=\GL_3(\C)/\C^{\times}$ preserving the
Hesse pencil:
\begin{equation}
    \label{Hesse group}
    \gm:=\big\{g\in\PGL : \text{ for any }\lambda,\ \ g\cdot F_\lambda = F_\mu  \text{ for some } \mu \big\}.
    %\text{for any } \lambda,\ \ \  g\cdot F_\lambda = F_\mu \text{ for some } \mu\}
\end{equation}
Here $g\cdot F$ denotes $F$ under a change of coordinates $g\cdot F(x,y,z):=F(g^{-1}(x,y,z))$. 

The flex points on any smooth  cubic plane curve $F$ in $\CP^2$ form an abstract configuration of 9 points and 12 lines where each point lies on 4 lines and each line contains 3 points, called the \emph{Hesse configuration}. In particular, there is a bijection taking points to points and lines to lines:
\begin{equation}
    \label{Hesse config}
    \{\text{flex points on $F$}\}\xrightarrow[]{\cong} \A^2(\F_3)
\end{equation}
Here the affine space $\A^2(\F_3)$ is the vector space $\F_3^2$ but without a choice of the origin. All curves in the Hesse pencil (\ref{pencil}) share the same set of 9 flex points. The natural action of the Hessian group $\gm$ on the left hand side of (\ref{Hesse config}) preserves lines. 
%We choose a flex $p_0=[1:-1:0]$ to be the identity of $F_\lambda$ as an elliptic curve. The remaining 8 flex points are precisely the 3-torsion points. Under such identification, the action of $\gm$ on the 9 flex points is through affine linear (i.e. a linear map composed with a translation) maps on $\F_3^2$. 
Hence, we have a group homomorphism below which turns out to be an isomorphism:
$$\gm\xrightarrow[]{\cong} \mathrm{Aut}(\A^2(\F_3))= \ASL\cong\F^2_3\rtimes \SL_2(\F_3).$$
In particular, $\gm$ is a finite group of order 216. We choose a flex $p_0=[1:-1:0]$ to be the identity of $F_\lambda$ as an elliptic curve. The remaining 8 flex points are precisely the 3-torsion points. After the choice, the affine bijection (\ref{Hesse config}) becomes an isomorphism of abelian groups. The subgroup $\go\le\gm$ fixing $p_0$ is mapped to the subgroup $\SL_2(\F_3)\le\ASL$ fixing the origin.

The Hessian group $\gm\subset\PGL$ can be  generated by the following four elements: 
\begin{equation}
    \label{matrices}
    A=
\begin{bmatrix}
     0 & 0 & 1 \\
    1 & 0 & 0 \\
    0 & 1 & 0 \\    
\end{bmatrix}
,\  B=
\begin{bmatrix}
     1 & 0 & 0 \\
    0 & w & 0 \\
    0 & 0 & w^2 \\    
\end{bmatrix} 
,\  C=
\begin{bmatrix}
     1 & 0 & 0 \\
    0 & 1 & 0 \\
    0 & 0 & w \\    
\end{bmatrix} 
,\  D=
\begin{bmatrix}
     w^2 & w & 1 \\
    w & w^2 & 1 \\
    1 & 1 & 1 \\    
\end{bmatrix} 
\end{equation}
where $w:=e^{2\pi i/3}$. In particular, $C$ and $D$ generate the subgroup $\go\cong \SL_2(\F_3)$, while $A$ and $B$ generate a normal subgroup isomorphic to $\F_3^2$ which acts freely and transitively on $\A^2(\F_3)$ as translations. 

% If we consider a smooth cubic plane curve as an elliptic curve with a flex point as identity, then the remaining eight flex points are precisely 3-torsion points. Hence, the action of $\gm$ on the 9 common flex points of curves $F_\lambda$ in the Hesse pencil is through affine maps  This gives the following map which turns out to be an isomorphism:
% $$\gm\xrightarrow[]{\cong} \ASL(2,3)$$
% where $\ASL(2,3)\cong(C_3\times C_3)\rtimes \SL(2,3)$ 

% Elements in $\PGL$ that preserve the Hesse pencil form a finite group of order 216, called the \emph{Hesse group}. 

% Every smooth cubic plane curve is projectively equivalent to a curve in the \emph{Hesse canonical form}:
% $$F_\lambda : x^3+y^3+z^3-3\lambda xyz=0$$
% for some $\lambda\in \C$ such that $\lambda^3\ne 1$. The automorphism group of this family
% \begin{equation}
%     \label{Hesse group}
%     \gm:=\{g\in\PGL\ :\ \forall \lambda 
% \end{equation}
% , namely the group $\gm$ consisting of all $g\in \PGL$ that sends each $F_\lambda$ to some other $F_\mu$, is called is called the \emph{Hesse group}. 
% Moreover, all curves in this family $F_\lambda$ where $\lambda^3\ne 1$ have exactly the same 9 flex points. See \emph{e.g.} Lemma 1 in \cite{AD} for proofs. The family of curves $F_\lambda$ is called the \emph{Hesse pencil}. Moreover, $F_\lambda$ is smooth if and only if $\lambda^3\ne 1$. 

\section{Upper bound for Schwarz genus}
In this section, we prove the upper bound in Theorem \ref{main}.
\begin{thm}
\label{upper bound}
    $g(\xf/\x)\le 9$.
\end{thm}

% Recall that $$\xa := \{ (F,p_1,\cdots,p_9): p_i  \textrm{'s are all the flex points on the curve defined by } F \}$$
% is a Galois $S_9$-cover of $\x$. This cover is disconnected because its monodromy group, preserving lines of flex points, is a proper subgroup $\ASL$ of $S_9$. One can define a connected component of $\xa$ as follows:
% %Given any smooth cubic plane curves, we know that its 9 flex points form an abstract configuration of 9 points and 12 lines in $\CP^2$. 
% $$\xg:=\{ (F,f): f  \textrm{ is an affine bijection from $\A^2(\F_3)$ to the flex points on $F$}\}.$$
% $\xg$ is a Galois $\ASL$-cover of $\x$. The monodromy group $\ASL$ acts freely on $\xg$ by pre-composition with $f$. Since this action is also transitive, $\xg$ is connected. We can define  an inclusion $\xg\hookrightarrow\xa$ by fixing a particular order of points in $\A^2(\F_3)$. Thus, $\xg$ is diffeomorphic to a connected component of $\xa$. Moreover, $\xa$ is a disjoint union of several copies of $\xg$, the number of which is equal to the subgroup index of $\ASL\subseteq S_9$. As a consequence, we have
% $$g(\xa/\x)=g(\xg/\x).$$

% %here!!!!!!!!!!!!!!!!!!!

\begin{lem}
Let $\He:=\{\lambda\in\C\ |\ \lambda^3\ne1\}$ denote the space that parameterizes smooth curves in the Hesse pencil (\ref{pencil}). The Hessian group $\gm$ acts on $\He$ as its automorphism group and acts freely on $\PGL$ by multiplication from the right. There is a diffeomorphism
$$\x \cong\PGL\underset{\gm}{\times} \He$$
where the right hand side denotes the quotient space by the diagonal action of $\gm$ on the product.  
% We have the following commutative diagram
% \[\begin{tikzcd}
% 	\xf && \PGL\underset{\go}{\times} \He \\
% 	\x && \PGL\underset{\gm}{\times} \He
% 	\arrow["{}"', from=1-1, to=2-1]
% 	\arrow["{}", from=1-3, to=2-3]
% 	\arrow["\cong"', from=2-3, to=2-1]
% 	\arrow["\cong"', from=1-3, to=1-1]
% \end{tikzcd}\]
% where the horizontal maps are diffeomorphisms, and the vertical maps are 9-to-1 covers.
\end{lem}
\begin{proof}
Consider the map 
$$\phi: \PGL\times \He\longrightarrow \x$$
given by $\phi(g,\lambda)=g\cdot F_\lambda$. This map is surjective since every smooth cubic plane curve is projectively equivalent to a curve in the Hesse pencil. The fibers of $\phi$ are precisely orbits of the Hessian group $\gm$ acting diagonally on the product. This action is free since $\gm$ acts freely on $\PGL$ as a subgroup. In particular, $\phi$ is a Galois $\gm$-cover and induces the diffeomorphism in the lemma.
\end{proof}

Similarly, we can define a map 
$\psi: \PGL\times \He\longrightarrow \xf$
by $\psi(g,\lambda)=(g\cdot F_\lambda,\  g\cdot p_0)$ where $p_0=[1:-1:0]$ is one of the common flex points of curves in the Hesse pencil. The fibers of $\psi$ are orbits of $\go$, the subgroup of $\gm$ fixing $p_0$, acting diagonally on the product. %Hence, we get a diffeomorphism $\xf\cong \PGL\underset{\go}{\times} \He$. 
Hence, we obtain the following proposition.

\begin{prop}
    \label{commutative diagram}
    The following diagram commutes
\[\begin{tikzcd}
	\xf && \PGL\underset{\go}{\times} \He \\
	\x && \PGL\underset{\gm}{\times} \He
	\arrow["{}"', from=1-1, to=2-1]
	\arrow["{}", from=1-3, to=2-3]
	\arrow["\phi"',"\cong", from=2-3, to=2-1]
	\arrow["\psi"', "\cong", from=1-3, to=1-1]
\end{tikzcd}\]
where the horizontal maps are diffeomorphisms, and the vertical maps are 9-to-1 covers.
\end{prop}

\begin{proof}[Proof of Theorem \ref{upper bound}]
By Proposition \ref{commutative diagram}, it suffices to show that the cover
\begin{equation}
    \label{product cover}
    \PGL\underset{\go}\times \He\to\PGL\underset{\gm}{\times} \He
\end{equation}
has Schwarz genus $\le 9$. Consider the natural projection 
$$p:\PGL\underset{\gm}{\times} \He \to \PGL/\gm$$ 
given by projecting onto the  first coordinates. It is straightforward to check that the cover in (\ref{product cover}) is the pullback of the cover
$$\PGL/\go\to \PGL/\gm$$
via the map $p$. Since Schwarz genus is nonincreasing under pullback (Proposition \ref{pullback SG}), it now suffices to show that 
$$g(\PGL/\go\to \PGL/\gm)\le 9.$$
Finally, observe that $\PGL$ deformation retracts onto the projective unitary group $\PU=\U_3/Z(\U_3)$. 
We have $\gm\subseteq\PU$ since $\gm$ is generated by projective unitary matrices shown in (\ref{matrices}). Therefore, $\PGL/\gm$ is homotopy equivalent to $\PU/\gm$ which is a closed manifold of real dimension 8. By Corollary \ref{coho dim}, we have 
$$g(\PGL/\go\to \PGL/\gm)\le 1+8=9.$$
\end{proof}

% Denote $P:=\CP^1\setminus\{\infty,-3,-3e^{\frac{2\pi\ii}{3}},-3e^{\frac{4\pi\ii}{3}}\}$ and $F_{\lambda}(x,y,z)=x^3+y^3+z^3+\lambda xyz$, define 
% $\Gamma_{216}:=\{g\in PGL(3,\C)|\forall \lambda\in P,\exists \mu\in P\text{ such that }g\cdot F_{\lambda}=F_{\mu}\}$.

% \begin{lem}[Lemma 1 of \cite{Artebani}]
% Every smooth cubic curve $F\in Y$ is $PGL(3,\C)$ equivalent to $F_{\lambda}$ for some $\lambda\in P$.
% \end{lem}
% By this lemma, there is a surjective map 
% \bea
% PGL(3,\C)\times P&\to& Y\nn\\
% (g,\lambda)&\mapsto& g\cdot F_{\lambda}.
% \eea
% So $Y=PGL(3,\C)\times_{\Gamma_{216}}P.$

\section{Lower bound for Schwarz genus}
In this section,  we prove the lower bound in Theorem \ref{main}.

\begin{thm}
\label{lower bound}
    $g(\xf/\x)\ge 8$
\end{thm}

Let $K$ denote the subgroup of $\gm$ generated by the two elements $A,B\in\gm$ as in (\ref{matrices})
\begin{equation}
    \label{K}
    K:=\langle A,B\rangle.
\end{equation}
Notice that $K$ is abstractly isomorphic to a product of cyclic groups $C_3\times C_3$. 

\begin{prop}\label{homological-genus}
Let $BK$ denote the classifying space of $K$. Consider the classifying map $cl: \PU/K\to BK$ associated to the principal $K$-cover $\PU\to\PU/K$. 
\begin{enumerate}
    \item The following induced map is nonzero:
    $$H^7(BK;\F_3)\xrightarrow[]{cl^*} H^7(\PU/K;\F_3)$$
    \item In particular, $g(\PU\to \PU/K)\ge8.$
\end{enumerate}
\end{prop}
\begin{proof}[Proof Proposition \ref{homological-genus} $\Rightarrow$ Theorem \ref{lower bound}]
Consider the Fermat cubic curve
$$F_0(x,y,z)=x^3+y^3+z^3.$$
The subgroup $K=\langle A,B\rangle$ preserves $F_0$. Hence, the map $g\mapsto g\cdot F_0$ gives a well-defined map
$$\phi:\PU/K\to \x.$$
Moreover, we claim that the pullback of the cover $\xf\to\x$ along the map $\phi$ is exactly the cover $\PU\to \PU/K$. This is because $K\cong C_3\times C_3$ acts on the 9 flex points of $F_0$ freely and transitively as translations of $\A^2(\F_3)$.
% Moreover, by direct computation, one can check that $K\cong C_3\times C_3$ acts on the 9 flex points of $F_0$ freely and transitively as translations of $\A^2(\F_3)$. Hence, the pullback of the cover $\xf\to\x$ along the map $\phi$ is exactly the cover $\PU\to \PU/K$. 
By Proposition \ref{pullback SG}, we have
    $$g(\xf\to\x)\ge g(\PU\to\PU/K)\ge 8.$$
\end{proof}

% The proof of Proposition \ref{homological-genus} is the most technically challenging part of this paper and will take the rest of this section.
\begin{proof}[Proof of Proposition \ref{homological-genus}]
    (2) directly follows from (1) by Theorem \ref{homo genus}. We focus on (1).

    The first step is to factor the classifying map $cl:\PU/K\to BK$ into a composition of two maps. Define $\pt:=p^{-1}(K)$ where $p:\U_3\to\PU$ in the natural projection. Then we have a central extension of groups 
    \begin{equation}
        \label{central extension}
            1\to \U_1\to \pt\xrightarrow[]{p} K\to1.
    \end{equation}
    Notice that $\PU/K$ and $\U_3/\pt$ are the same space by definitions. Then we have the following commutative diagram:
    \begin{equation}
        \label{factor}
        \xymatrix{
\H^i(BK;\F_3)\ar[r]^{cl^*}\ar[d]_{p^*}&\H^i(\PU/K;\F_3)\ar[d]^{=}\\
\H^i(B\pt;\F_3)\ar[r]^{\phi}&\H^i(\U_3/\tilde{P};\F_3),}
    \end{equation}
where  $B\pt$ denotes the classifying space of the topological group $\pt\le \U_3$. To show that $cl^*$ is nonzero when $i=7$, we will show that the image of $p^*$ is not contained in the kernel of $\phi$.

The central extension (\ref{central extension}) induces a fibration $B\U_1\to B\pt\xrightarrow[]{p} B K$. The Serre spectral sequence associated to this fibration had been worked out by Leary \cite{Leary}. 
\begin{thm}[Leary, Theorem 2 in  \cite{Leary}]
\label{Leary}
    $H^*(B\pt;\F_3)$ is generated by elements $y,y',x,x',c_2,c_3$ of degrees $1,1,2,2,4,6$, respectively,  
    % where
    % $$\deg(y)=\deg(y')=1,\ \deg(x)=\deg(x')=2,\ \deg(c_2)=4,\ \deg(z)=6$$
    subject to the following relations:
    \begin{align*}
        xy'&=x'y&  x^3y'&=x'^3y\\
        c_2y&=-x^2 y &c_2y'&=-x^2 y'\\
        c_2x&=-x^3 &c_2x'&=-x'^3\\
        c_2^2&=x^4+x'^4-x^2x'^2.&&
    \end{align*}
Moreover, $c_2$ and $c_3$ are the pullbacks of the second and third Chern classes in $H^*(B\U_3;\Z)$ via the inclusion $\pt\hookrightarrow \U_3$ and  modulo 3. The first Chern class pulls back to $c_1=yy'.$
\end{thm}
Recall that $H^*(BC_3;\F_3)$ is a free graded algebra generated by $\tilde{y}$ and $\tilde{x}$ of degree 1 and 2, respectively. Since $K\cong C_3\times C_3$, we can find generators $\tilde{y},\tilde{y}',\tilde{x},\tilde{x}'$ for $H^*(\B K;\F_3)$ by the K\"unneth formula. By Leary's proof of Theorem \ref{Leary} in \cite{Leary}, the map $p^*$ in (\ref{factor}) sends the generators $\tilde{y},\tilde{y}',\tilde{x},\tilde{x}'$ to $y,y',x,x'$, respectively. Hence, we have the following presentation of subring:
$$p^*\big(H^*(BK;\F_3)\big) = \langle y,y',x,x'\rangle /(xy'-x'y,x^3y'-x'^3y)$$
where the quotient ideal is given by the first two relations in Theorem \ref{Leary}. In particular, on $H^7$, we have:
\begin{equation}
    \label{image of p}
    p^*\big(H^7(BK;\F_3)\big)=\mathrm{span}\{yx^3,yx^2x',yxx'^2,yx'^3,y'x^3,y'x^2x',y'xx'^2,y'x'^3\}.
\end{equation}
Those elements above spanned $p^*\big(H^7(BK;\F_3)\big)$, but subject to linear relations in the ideal $(xy'-x'y,x^3y'-x'^3y)$.

Next, we study the kernel of $\phi$ in (\ref{factor}). We first recall some basic facts about the Serre spectral sequence of the universal principal $\U_3$-bundle:
\begin{equation}
\label{EG}
    \U_3\to E\U_3\to B\U_3
\end{equation}
By Borel's Transgression Theorem, each generator $a_{2i-1}$ of $H^{2i-1}(\U_3;\Z)$ transgresses to the $i$-th Chern class in $H^{2i}(B\U_3;\Z)$ for $i=1,2,3$. 

Now, the fibration
\begin{equation}
    \label{pt}
    \U_3\to\U_3/\pt\to B\pt
\end{equation}
is homotopy equivalent to the pullback of the universal bundle (\ref{EG}) along a classifying map $B\pt\to B\U_3$. Consequently, in the (cohomological) Serre spectral sequence of (\ref{pt}), each generator $a_{2i-1}$ of $H^{2i-1}(\U_3;\F_3)$ transgresses to  $c_i\in H^{2i}(B\pt;\F_3)$ as in Theorem \ref{Leary}. Hence, the kernel of $\phi:H^*(B\pt;\F_3)\to H^*(\U_3/\pt;\F_3)$ is the ideal generated by $c_1,c_2,c_3.$ In particular, on $H^7$, we have
$$\text{Ker}\phi=c_1\H^5(B\tilde{P};\F_3)+c_2\H^3(B\tilde{P};\F_3)+c_3\H^1(B\tilde{P};\F_3).$$

% % which is the $i$-th Chern class of the principal $\U_3$-bundle (\ref{pt}), for $i=1,2,3$. 
% Moreover, Leary (Theorme 2, \cite{Leary}) computed those Chern classes in terms of generators of $H^*(B\pt;\F_3)$ as 
% $c_1=yy'$ and $c_3=z$. 

By the relations in Theorem \ref{Leary}, we have that
\bea
\ \ \ \ \ \  \ \ \ \ \ c_1\H^5(B\tilde{P};\F_3)=(yy') \cdot\mathrm{span}\{ yx^2,yxx',yx'^2,c_2y,y'x^2,y'xx',y'x'^2,c_2y'\}=0,
\eea
\bea
c_2\H^3(B\tilde{P};\F_3)=c_2\cdot \mathrm{span}\{  yx,yx'=y'x,y'x'\}
\eea
where
\bea
c_2yx&=&-x^2yx=-yx^3,\nn\\
c_2y'x'&=&-x'^2y'x'=-y'x'^3,\nn\\
c_2yx'&=&-x^2yx'=-yx^2x',
\eea
and
\bea
c_3\H^1(B\tilde{P};\F_3)=c_3\cdot \mathrm{span}\{  y,y'\}.
\eea
In particular, we find an element 
$yxx'^2=y'x^2x'\in\text{Im}p^*$ but $yxx'^2=y'x^2x'\not\in\text{Ker}\phi$. Hence, $\text{Im}p^*\not\subseteq\text{Ker}\phi$. Therefore, $cl^*:\H^7(BK;\F_3)\to \H^7(\PU/K;\F_3)$ is nonzero.

\end{proof}

\section{Topological complexity of flex-finding algorithms}
\label{tree}
Following Smale \cite{Smale}, we consider an \emph{algorithm} to be a finite rooted tree, consisting of a root for the input, leaves for the output, and internal nodes of the following two types:
\begin{itemize}
    \item \emph{Computation nodes}\ \ \begin{tikzpicture}[scale=0.3]
   
    \draw (0,0) circle(.2);
    \draw[->] (0,1)-- (0,0.1);
    \draw[->] (0,-0.1)-- (0,-1);
   
    \end{tikzpicture}\ \ , which receives a sequence of real numbers, computes using rational operations $+,\ -,\ \times,\ \div\ $, and gives another sequence of real numbers to the next node;
    
    \item \emph{Branching nodes}\ \ \begin{tikzpicture}[scale=0.3]
   
    \draw (0,0) circle(.2);
    \draw[->] (0,1)-- (0,0.1);
    \draw[->] (-0.08,-0.08)-- (-0.8,-0.8);
    \draw[->] (0.08,-0.08)-- (0.8,-0.8);
    \end{tikzpicture}\ \ , which go right or left according to whether an inequality $h\le 0$ is true or false, where $h$ is one of the real numbers in the sequence which the node receives.
\end{itemize}
An algorithm is said to \emph{solve} the following problem
\begin{quote}
$\pe$: given a smooth cubic plane curve, find a point in $\CP^2$ that is within $\epsilon$ distance from a flex point of the given curve. 
\end{quote}
if it inputs a sequence of real numbers recording the real and imaginary parts of the coefficients of $F(x,y,z)$ defining an arbitrary cubic plane curve and outputs a sequence of real numbers recording the real and imaginary parts of the coordinates of a point $[x:y:z]$ in $\CP^2$ satisfying the requirement. The \emph{topological complexity} of an algorithm is the number of branching nodes in the tree. The {topological complexity of the problem $\pe$} is the minimum of the topological complexity of all algorithms solving the problem.

% Consider the problem
% \begin{quote}
% $\pe$: given a smooth cubic plane curve, find a point in $\CP^2$ that is within $\epsilon$ distance from a flex point of the given curve. 
% \end{quote}
% %More precisely, given a smooth cubic plane curve defined by $F$, we want to find a point $p\in \CP^2$ that is within $\epsilon$ distance from a flex point of $F$. 
% Here distance is measured by the Fubini-Study metric, a Riemannian metric on $\CP^2$ invariant of the $\PGL$-action. 

% Following Smale \cite{Smale}, we define an \emph{algorithm} to be a finite rooted tree, consisting of one root for the input, leaves for the output, and internal nodes of the following two types

% \begin{itemize}
%     \item \emph{Computation nodes}, which compute using rational operations $+,\ -,\ \times,\ \div$. 
%     \item \emph{Branching nodes}, which go right or left according to whether an inequality is true or false.
% \end{itemize}
% We say an algorithm solves the problem $\pe$ if it inputs the coefficients of $F(x,y,z)$ defining a cubic plane curve, and outputs coordinates of a point satisfying the requirement. The inputs and outputs are in terms of their real and imaginary parts. The \emph{topological complexity of an algorithm} is the number of branching nodes in the tree. The \emph{topological complexity of the problem $\pe$} is the minimum of the topological complexity of all algorithm solving the problem.

% \begin{thm}
%     \label{complexity}
%     The topological complexity of the problem $\pe$ is at least 7 for any $\epsilon<\pi/3$.
% \end{thm}
We now prove Theorem \ref{top complexity} which gives a lower bound for the topological complexity of any algorithm that finds a flex point on cubic plane curves. Theorem \ref{top complexity} directly follows from Theorem \ref{lower bound} together with the following proposition. 
\begin{prop}
    \label{Smale principle}
    For  any $\epsilon<\pi/3$, the topological complexity of the problem $\pe$ is at least $g(\xf/\x)-1$.
\end{prop}
This proposition is a direct analog of the  ``Smale principle" as Vassiliev calls it (Section 2.3 in \cite{Vassiliev}). Specifically, Smale proved that the topological complexity of the problem of finding all roots within $\epsilon$ of a complex polynomial $f(z)$ is bounded from below by the Schwarz genus of the cover associated to the problem minus one, for all $\epsilon$ sufficiently small (See Theorem A in \cite{Smale}). After his proof, Smale remarked that the argument works in ``considerably greater generality". Our proof of Proposition \ref{Smale principle} below is a slight modification of Smale's argument.  Even better, we can get an explicit bound $\pi/3$ for how small $\epsilon$ needs to be. 

\begin{proof}
It suffices to show that any algorithm  with $k$ leaves (and hence with $k-1$ branching nodes) solving the problem $\pe$ gives an open cover of $\x$ with $k$ subsets and a section of the cover $\xf\to\x$ defined on each subset. 

The explicit upper bound $\pi/3$ comes from the following observation: the distance between any two flex points on an arbitrary smooth cubic plane curve is exactly $2\pi/3$. To see this, it suffices to consider only those curves $F_\lambda$ in the Hesse pencil (\ref{pencil}) because every cubic curve is projective equivalent to a curve in the Hesse pencil and the Fubini-Study metric on $\CP^2$ is $\PGL$-invariant. Moreover, any pair of flex points of $F_\lambda$ has the same distance because the action of $\gm\subset\PGL$ on the 9 flex points is 2-transitive, \emph{i.e.} transitive on pairs of distinct elements. Finally, the distance between two flex points of $F_\lambda$, for example, $[1:-1:0]$ and $[1:-e^{2\pi/3}:0]$, is $2\pi/3$.

% First of all, we show that $\x$ deformation retracts onto a compact subspace $K=\PU\underset{\gm}\times T$ where $T$ is the 1-skeleton of the regular tetrahedron.  

Suppose we are given an algorithm solving $\pe$ with output leaves numbered $i=1,\cdots, k$. Let $V_i$ denote the subset of $\x$ consisting of all those inputs that will arrive at the $i$-th output leaf of the algorithm. Then $\x$ as a set is a disjoint union of $V_i$'s. Moreover, each $V_i$ is the intersection of a closed subset in $\x$, consisting of inputs satisfying $h\le 0$ at each branching node where the path from the input root to the $i$-th output leaf goes to the right, and an open subset in $\x$, consisting of inputs not satisfying $h\le 0$ at each branching node where the same path goes to the left. The algorithm gives a map $\phi_i:V_i\to \CP^2$ for each $i=1,\cdots,k$ such that for any $F\in V_i$, the point $\phi_i(F)$ is $\epsilon$-close to a flex point of $F$. By the Tietze Extension Theorem, $\phi_i$ can be extended to an open subset $U_i\supseteq V_i$ satisfying the same property: any $F\in U_i$, the point $\phi_i(F)$ is $\epsilon$-close to a unique flex point of $F$, which we denote by $s_i(F)$. The flex point $s_i(F)$ is unique because $\epsilon<\pi/3$, which is one half of the distance between any two distinct flex points on $F$. Now we have an open cover $U_i$'s of $\x$ by $k$ many subsets, on which there is a section $s_i:U_i\to \xf$ of the cover $\xf\to\x$. 
\end{proof}

\begin{remark}
Unlike Smale's proof in \cite{Smale}, our proof works in greater generality if we allow algorithms to have computation nodes by continuous functions, rather than by rational functions only.
\end{remark}

\section{What is left to determine the Schwarz genus?}

Theorem \ref{main} shows that $g(\xf/\x)$ is either 8 or 9. What do we need to determine which one is the case?  

Applying Theorem \ref{fiberwise join} to the cover $\xf\to \x$, we see that $g(\xf/\x)\le 8$ if and only if the fiber bundle 
\begin{equation}
    \label{8 join}
    F^{*8}\to E_8\to\x
\end{equation}
has a continuous section where $F$ denotes the fiber of $\xf\to\x$. Since $F$ is just a set of 9 points, $F^{*8}$ is a wedge of 7-spheres. The first obstruction to a section of the bundle above is 
$$\mathfrak{o}_8\in H^8(\x;\pi_{7}(F^{*8})).$$
Moreover, one can show that $$\pi_{7}(F^{*8})\cong \widetilde{H}_0(F;\Z)^{\otimes 8} \cong(\Z^8)^{\otimes 8}.$$

\begin{prop}
    $g(\xf/\x)=8$ if and only if the first obstruction $\mathfrak{o}_8$ is zero.
\end{prop}
\begin{proof}
The ``only if" part follows from obstruction theory. To prove the ``if" part, we need to show that there can be no further obstruction after the first one. However, we proved in Proposition \ref{commutative diagram} that the cover $\xf\to\x$ is a pullback of a cover $\PU/\go\to \PU/\gm$. Hence, those obstruction classes on $\x$ are pullbacks of the corresponding obstruction classes on $\PU/\gm$. Now recall that $\PU/\gm$ is an 8-dimensional manifold and thus has no nonzero cohomology class beyond dimension 8. 
\end{proof}
We are not able to determine whether $\mathfrak{o}_8$ is zero or not. One difficulty is that the coefficient module $(\Z^8)^{\otimes 8}$ is too big for explicit calculation by hand or by a computer. However, we can prove that $\mathfrak{o}_8$ is a 3-torsion element. 

\begin{prop}
\label{3-torsion}
    $3\mathfrak{o}_8=0$ in $H^8(\x;\pi_{7}(F^{*8}))$.
\end{prop}
\begin{lem}
\label{transfer}
Let $Z\to E\to B$ be a fiber bundle and let $\mathfrak{o}_i\in H^i(B;\pi_{i-1}(Z))$ be the first obstruction class to a section of it. If there is a degree-$d$ cover $\tilde{B}$ of the base $B$ such that the pullback of the bundle to $\tilde{B}$ has a continuous section, then $d\mathfrak{o}_i=0$.
\end{lem}
\begin{proof}
Let $p:\widetilde{B}\to B$ denote the covering map. We know that $p^*\mathfrak{o}_i=0$ since the pullback bundle has a section. Now, we have a composition of maps
\bea
H^i(B;\pi_{i-1}(Z))\xrightarrow[]{p^*}& H^i(\widetilde{B};\pi_{i-1}(Z))&\xrightarrow{{tr}}H^i(B;\pi_{i-1}(Z))\nn\\
\mathfrak{o}_i\mapsto& p^*\mathfrak{o}_i=0&\mapsto d\mathfrak{o}_i
\eea
where $tr$ denotes the transfer map. Hence, $d\mathfrak{o}_i=0$.
\end{proof}
\begin{proof}[Proof of Proposition \ref{3-torsion}]
The pullback of the fiber bundle (\ref{8 join}) along the degree-$9$ cover $\xf\to\x$ has a continuous section given by the natural inclusion of fibers $F\to F^{*8}$. Lemma \ref{transfer} implies that $9\mathfrak{o}_8=0$.

We claim that to prove Proposition \ref{3-torsion}, it suffices to prove the following 
\begin{quote}
    \textbf{Claim:} There is a degree-24 covering $p:Y\to \x$ such that the pullback of the covering $\xf\to\x$ along $p$ has Schwarz genus $\le8$. 
\end{quote}
Assuming the claim, we have that $24\mathfrak{o}_8=0$ by Lemma \ref{transfer}. Then we have that $3\mathfrak{o}_8 = 3\cdot(9\mathfrak{o}_8)-24\mathfrak{o}_8=0.$

Now we focus on proving the claim. Again, recall from our proof of Proposition \ref{commutative diagram} that the cover $\xf\to\x$ is a pullback of the cover 
\begin{equation}
    \label{PU}
    \PU/\go\to\PU/\gm
\end{equation}
along a map $\x\to\PU/\gm$. Since Schwarz genus is nonincreasing under pullback (Proposition \ref{pullback SG}), in order to prove the claim stated above, it suffices to prove the same claim where we replace $\x$ by $\PU/\gm$ and $\xf$ by $\PU/\go$. Next, we consider the degree-24 covering map
$$p:\PU/K\to \PU/\gm$$
where again $K$ denotes the subgroup defined in (\ref{K}). The cover has degree 24 because $|\gm|=216$ and $|K|=9$. Since $K\cap\go=1$, the pullback of the cover (\ref{PU}) along $p$ is the cover 
$$\PU\to\PU/K.$$
We are done if we can show the following:
\begin{prop}
\label{K upper bound}
    $g(\PU\to\PU/K)\le8$.
\end{prop}
\begin{proof}
Our strategy is to find a closed subgroup $H\cong \U_1$ inside $\PU$ such that the action of $K$ from the right on the (left) coset space $H\backslash \PU$ is free. If we can find such $H$, then we have the following pullback diagram of covers:
\[\begin{tikzcd}
	\PU & H\backslash\PU \\
	\PU/K & H\backslash\PU/K 
	\arrow[from=1-1, to=1-2]
	\arrow[from=1-1, to=2-1]
	\arrow[from=1-2, to=2-2]
	\arrow[from=2-1, to=2-2]
\end{tikzcd}\]
Then Proposition \ref{K upper bound} will follow from Proposition \ref{pullback SG} and Corollary \ref{coho dim}, together with the observation that the double coset space $H\backslash\PU/K$ is a manifold of dimension $8-1=7$.

Now let's find such $H\le \PU$. We just need $H$ to satisfy that
$$H\cap gKg^{-1}=1,\ \ \ \ \forall g\in\PU.$$
We claim that the subgroup
$$H:=\{\mathrm{diag}(1,z,z)\ |\ z\in \U_1\}$$
suffices. Even though eigenvalues of elements in $\PU$ are only well-defined up to simultaneous multiplication of a scalar, the total number of distinct eigenvalues is well-defined for elements in $\PU$. By a straightforward computation, we  check that every nontrivial element  in $K=\langle A,B\rangle$ has 3 distinct eigenvalues. However, every nontrivial element in $H$ has 2 distinct eigenvalues. Hence, $H$ intersects trivially with any conjugate of $K$ in $\PU$.
\end{proof}

\end{proof}

\begin{remark}
Is it possible to improve our arguments above to get better lower or upper bounds for $g(\xf/\x)$? It seems not for the following reasons. 

We obtained the lower bound $g(\xf/\x)\ge 8$ in Theorem \ref{lower bound} by showing that $g(\PU\to\PU/K)\ge 8$ in Proposition \ref{homological-genus}. Now Proposition \ref{K upper bound} tells us that our previous argument cannot be improved to obtain any stronger lower bound.

In the opposite direction, one might hope to improve the upper bound $g(\xf/\x)\le 9$ in Theorem \ref{upper bound} by adapting our proof of Proposition \ref{K upper bound} to the cover $\PU/\go\to\PU/\gm$ if we replace $K$ by the larger group $\gm$. This is not possible because one can show that any closed subgroup $H$ in $\PU$ will have a nontrivial intersection with some conjugate of $\gm$. 
\end{remark}

 \bibliographystyle{abbrv}
\bibliography{main.bib}

\end{document}